\documentclass{amsart}
\usepackage{amssymb,latexsym,amsthm,amsmath}
\usepackage{hyperref}
\usepackage{commath}
\usepackage{marginnote}
\numberwithin{equation}{section}

\newtheorem{theorem}{Theorem}[section]
\newtheorem{lemma}[theorem]{Lemma}
\newtheorem{proposition}[theorem]{Proposition}
\newtheorem{corollary}[theorem]{Corollary}

\theoremstyle{definition}

\theoremstyle{remark}
\newtheorem{remark}[theorem]{Remark}

\newcommand{\End}{\mathrm{End}}
\newcommand{\U}{\mathrm{U}}
\newcommand{\SU}{\mathrm{SU}}
\newcommand{\GL}{\mathrm{GL}}

\newcommand{\C}{\mathbb{C}}
\newcommand{\R}{\mathbb{R}}
\newcommand{\Z}{\mathbb{Z}}
\newcommand{\N}{\mathbb{N}}
\newcommand{\T}{\mathbb{T}}
\newcommand{\B}{\mathbb{B}}

\newcommand{\mH}{\mathcal{H}}
\newcommand{\mP}{\mathcal{P}}

\begin{document}

\title[Separately Radial and Radial Toeplitz and Representation Theory]{Separately Radial and Radial Toeplitz Operators on the Unit Ball and Representation Theory}

\author{Raul Quiroga-Barranco}
\address{Centro de Investigaci\'{o}n en Matem\'{a}ticas, Guanajuato, Mexico}
\email{quiroga@cimat.mx}

\dedicatory{To Sergei Grudsky on the occasion of his 60th birthday.}

\thanks{Research supported by SNI and a Conacyt Grant}

\begin{abstract}
    We study Toeplitz operators with separately radial and radial symbols on the weighted Bergman spaces on the unit ball. The unitary equivalence of such operators with multiplication operators on $\ell^2$ spaces was previously obtained by analytic methods in \cite{QVReinhardt} and \cite{GKV}, respectively. We prove that the same constructions can be performed with a purely representation theoretic approach to obtain the same conclusions and formulas. However, our method is shorter, more elementary and more elucidating.
\end{abstract}

\maketitle

\section{Introduction}\label{sec:Introduction}
The weighted Bergman spaces on the unit ball, as well as in other bounded domains, are of fundamental importance in analysis (see for example \cite{Zhu}). This is very much due to the existence of a reproducing kernel for Bergman spaces, which allows for Toeplitz operators to be considered naturally.

It has been found very useful to study Toeplitz operators whose symbols have special symmetries. With this respect, a remarkable work is \cite{GKV} where the Toeplitz operators with radial symbols (see Section~\ref{sec:Toeplitz_radial} for the definition) were proved to generate a commutative $C^*$-algebra. The proof was based on the construction of a Bargmann type transform that allows to simultaneously diagonalize the Toeplitz operators into multiplication operators over an $\ell^2$ space. Similar results with the same approach were obtained in \cite{QVReinhardt} for Toeplitz operators with separately radial symbols (see Section~\ref{sec:Toeplitz_separately_radial}), thus exhibiting commutative $C^*$-algebras generated by such operators.

The existence of commutative $C^*$-algebras generated by Toeplitz operators was further extended to the unit ball in \cite{GQVJFA}, \cite{QVBall1} and \cite{QVBall2}. Moreover, these works made pretty clear the importance of the Lie subgroups of the biholomorphism group of the corresponding domain. They also allowed to have a better understanding of the commutative $C^*$-algebras generated by Toeplitz operators: a classification was given in the case of the unit disk and several non-trivial examples were constructed on the unit ball.

Recently, in \cite{DOQJFA} it was established the existence of several types of symbols on every bounded symmetric domain for which the Toeplitz operators generate commutative $C^*$-algebras. Most of such symbols were given as invariant functions with respect to symmetric subgroups of the biholomorphism group of the corresponding domains, and the cases presented in \cite{DOQJFA} include the separately radial and radial symbols. The special role of the holomorphic discrete series associated to bounded symmetric domains (see Section~\ref{sec:Preliminaries} for the definition) was also put forward. The proofs of the results from \cite{DOQJFA} were based on the existence of multiplicity-free restrictions for the holomorphic discrete series. Such proofs can, in many cases, be traced back to the construction of Bargmann type transforms similar to the ones used in \cite{GKV,QVBall1,QVBall2}. In other words, there is an implicit relationship between the analytic and the representation theoretic approaches to the study of Toeplitz operators whose symbols have symmetries.

The goal of this work is to explicitly exhibit the relationship of the analytic and the representation theoretic approaches for the case of separately radial and radial symbols on the unit ball. We also explain how understanding such relationship allows to have better knowledge of these special types of Toeplitz operators.

A fundamental fact to keep in mind is that, with respect to the holomorphic discrete series, if a symbol is invariant under a subgroup, then the corresponding Toeplitz operator intertwines the action of the subgroup (see Proposition~\ref{prop:Toeplitz_intertwining}).

For the case of the separately radial symbols, that correspond to the subgroup $\T^n$ of the biholomorphism group of the unit ball $\B^n$, we define in Theorem~\ref{thm:End_Tn} a unitary operator $R$ that plays the role of the Bargmann type transform found in \cite{QVReinhardt}. However, its definition and the proof of its properties is purely representation theoretic. The arguments are in fact quite elementary based only on Schur's Lemma and the very basic facts of characters on tori. That this unitary map $R$ plays indeed the role of a Bargmann type transform is established in Theorem~\ref{thm:Toeplitz_separately_radial}, where we exhibit the simultaneous diagonalization of the Toeplitz operators with separately radial symbols into multiplication operators on $\ell^2(\N^n)$. Furthermore, we obtain the same expression for the functions that define the multiplication operators as those found in Theorem~10.1 from \cite{QVBall1}. Nevertheless, our proof is much more elementary and short. In fact, our representation theoretic approach allows to have a better understanding of the Toeplitz operators: we prove in Corollary~\ref{cor:orthogonality_relations_Tn} that the Toeplitz operators with separately radial symbols satisfy orthogonality relations.

A corresponding study is performed for radial symbols, for which the subgroup is $\U(n)$ with its linear action on $\B^n$. In Theorem~\ref{thm:End_U(n)} we consider the same unitary map $R$ from Theorem~\ref{thm:End_Tn} which turns out to provide a simultaneous diagonalization of the Toeplitz operators with radial symbols as found in \cite{GKV}. Our construction and the proof of the properties are again purely representation theoretic. This time we use Schur-Weyl duality but in its most simplest form. In any case, the main point is the fact that the spaces of homogeneous polynomials with a fixed degree are irreducible representations of the general linear group $\GL(n,\C)$. That $R$ plays the role of a Bargmann type transform is now proved in Theorem~\ref{thm:Toeplitz_radial}. But in this case the functions for the multiplication operators unitarily equivalent to Toeplitz operators with radial symbols are proved to be constant on the multi-indices with the same length, just as established in \cite{GKV}. This provides a function that belongs to $\ell^\infty(\N)$, as it was also observed in \cite{GKV}. However, with our representation theoretic approach we can provide an explanation to this behavior: a Toeplitz operator with radial symbol acts by a multiple of the identity on the spaces of homogeneous polynomials of the same degree. The point is that the latter are precisely the irreducible components of the representation of $\U(n)$ on every weighted Bergman space on the unit ball (see Proposition~\ref{prop:Poly_U(n)}). Finally, Theorem~\ref{thm:Toeplitz_radial} provides the same expression for the functions of the multiplication operators as those found in Theorem~3.1 from \cite{GKV}. We also prove in Corollary~\ref{cor:orthogonality_relations_U(n)} a set of orthogonality relations satisfied by the Toeplitz operators with radial symbols. Again, the proofs are shorter and more elementary than those found in \cite{GKV}.

\section{Preliminaries}\label{sec:Preliminaries}

\subsection{Intertwining operators on direct sums}
Let $H$ be a Lie group and $\pi$ a unitary representation on a Hilbert space $\mH$. We recall that a bounded operator $T : \mH \rightarrow \mH$ is called intertwining if and only if
\[
    T(\pi(h) v) = \pi(h) T(v),
\]
for every $v \in \mH$ and $h \in H$. In this case, we also say that $T$ intertwines the representation of $H$ and we denote by $\End_H(\mH)$ the algebra of such intertwining operators.

The following result is well known, but we present its proof for the sake of completeness.

\begin{proposition}
    \label{prop:End_H(mH)-commutative}
    Let $H$ be a Lie group and $\pi$ a unitary representation on a Hilbert space $\mH$. Suppose that $\mH$ contains a dense subspace that can be algebraically decomposed as
    \[
        V = \sum_{j \in J} \mH_j
    \]
    where the subspaces $\mH_j$ are mutually orthogonal, closed in $\mH$ and irreducible $H$-invariant modules. Then, the following conditions are equivalent
    \begin{enumerate}
      \item $\mH_{j_1} \not\cong \mH_{j_2}$, as $H$-modules for every $j_1 \not= j_2$,
      \item $\End_H(\mH)$ is commutative.
    \end{enumerate}
\end{proposition}
\begin{proof}
    First we note that we have
    \[
        \mH = \bigoplus_{j \in J} \mH_j
    \]
    as an orthogonal direct sum of Hilbert spaces.

    For every $j$, let $\pi_j : \mH \rightarrow \mH_j$ denote the orthogonal projection, and let $T \in \End_H(\mH)$ be given. Then, the map $\pi_{j_1} \circ T : \mH_{j_2} \rightarrow \mH_{j_1}$ is a homomorphism of $H$-submodules for every $j_1,j_2$. If we assume that (1) holds, then Schur's Lemma implies that $\pi_{j_2} \circ T|_{\mH_{j_1}} = 0$ whenever $j_1 \not= j_2$, and so that $T$ leaves invariant every subspace $\mH_j$. Applying Schur's lemma once more, we conclude that for every $j \in J$ the restriction $T : \mH_j \rightarrow \mH_j$ is a multiple of the identity. This proves that $\End_H(\mH)$ is commutative, thus showing that (1) implies (2).

    On the other hand, if there exists an isomorphism $T_0 : \mH_{j_1} \rightarrow \mH_{j_2}$ of $H$-modules for some $j_1 \not= j_2$, then Schur's lemma implies that the algebra $\End_H(\mH_{j_1}\oplus\mH_{j_2})$ consists of the maps of the form
    \begin{align*}
        \mH_{j_1}\oplus\mH_{j_2} &\rightarrow  \mH_{j_1}\oplus\mH_{j_2} \\
        (u,v) &\mapsto (au + b T_0^{-1}v, c T_0 u + d v),
    \end{align*}
    where $a,b,c,d \in \C$. In particular, $\End_H(\mH_{j_1}\oplus\mH_{j_2})$ is an algebra isomorphic to $M_2(\C)$. Extending by $0$ on $\bigoplus_{j \not= j_1, j_2} \mH_j$ it is clear that we have a natural inclusion
    \[
        \End_H(\mH_{j_1}\oplus\mH_{j_2}) \subset \End_H(\mH)
    \]
    of algebras. Hence, $\End_H(\mH)$ is not commutative. This now proves that (2) implies (1).
\end{proof}

We can improve the previous result to provide an interesting and useful description of the algebra $\End_H(\mH)$. Let us assume the hypotheses of Proposition~\ref{prop:End_H(mH)-commutative} and that its condition (1) holds. Choose $(e_l)_{l \in L}$ an orthonormal basis of $\mH$ for which we have a disjoint union
    \[
        L = \bigcup_{j \in J} L_j
    \]
    so that for every $j \in J$ the set $(e_l)_{l \in L_j}$ is an orthonormal base for $\mH_j$. Let us consider $R : \mH \rightarrow \ell^2(L)$ the isometry defined by
    \[
        R(v) = (\left<v, e_l\right>)_{l \in L},
    \]
    for every $v \in \mH$. Hence, its adjoint clearly satisfies
    \[
        R^*(x) = \sum_{l \in L} x_l e_l,
    \]
    for every $x \in \ell^2(L)$. It follows immediately that the map
    \begin{align*}
        \Phi : \End_H(\mH) &\rightarrow B(\ell^2(L)) \\
                T &\mapsto R T R^*
    \end{align*}
    is an injective homomorphism of algebras. This construction is analogous to the use of a Bargmann type transform as considered in \cite{GKV,GQVJFA,QVBall1,QVReinhardt} and allows us to obtain the following result.

\begin{proposition}
    \label{prop:End_H(mH)-realization}
    Suppose that the hypothesis of Proposition~\ref{prop:End_H(mH)-commutative} are satisfied and that its condition (1) holds. Then, with the above notation, every operator $T \in \End_H(\mH)$ is unitarily equivalent to $\Phi(T) = R T R^*$ which is the multiplication operator on $\ell^2(L)$ given by the function \begin{align*}
        \gamma_T : L &\rightarrow \C \\
            \gamma_T(l) &= \left< T(e_l), e_l \right>.
    \end{align*}
    Furthermore, the function $\gamma_T$ is constant on $L_j$ for every $j \in J$ and so induces a function $\widehat{\gamma}_T : J \rightarrow \C$ that belongs to $\ell^\infty(J)$ and that is given by
    \[
        \widehat{\gamma}_T(j) = \gamma_T(l)
    \]
    whenever $l \in L_j$. In particular, the map $\Phi$ realizes an isomorphism between the algebras $\End_H(\mH)$ and $\ell^\infty(J)$ given by the assignment
    \[
        T \mapsto \widehat{\gamma}_T.
    \]
\end{proposition}
\begin{proof}
    First we compute $RTR^*$ as follows. For every $x \in \ell^2(L)$, we have
    \begin{align*}
        \Phi(T)(x) &= RTR^*(x) \\
            &= R T \left( \sum_{l \in L} x_l e_l \right)
                = R \left( \sum_{l \in L} x_l T(e_l) \right) \\
        \intertext{and since $T$ acts by scalar multiplication on each $\mH_j$}
            &= R \left( \sum_{l \in L} x_l \left<T(e_l), e_l\right> e_l \right) \\
            &= \left( \left<T(e_l), e_l\right> x_l \right)_{l \in L}
                = \gamma_T x
    \end{align*}
    where $\gamma_T$ is the function defined in the statement.

    Next, we recall from the proof of Proposition~\ref{prop:End_H(mH)-commutative} that Schur's Lemma implies that $T$ is multiplication by a constant on each $\mH_j$, from which the claim involving the definition of $\widehat{\gamma}_T$ follows.

    On the other hand, for a given $x \in \ell^\infty(J)$ we can define the operator $T$ on $\mH$ by
    \[
        T|_{\mH_j} = x(j)Id_{\mH_j},
    \]
    for every $j \in J$. Then, it is easy to see that $T \in \End_H(\mH)$ and that $\widehat{\gamma}_T = x$. The isomorphism between $\End_H(\mH)$ and $\ell^\infty(J)$ is now clear.
\end{proof}

\begin{remark}
\label{rmk:gamma_T_eigenvalue}
    We note that in the statement above the number $\left< T(e_l), e_l \right>$ is the eigenvalue, say $\lambda_j$, of the action of $T$ on $\mH_j$ when $e_l \in \mH_j$. In particular, such value is the same for all $e_l$ that belong to $\mH_j$. In fact, for every unitary vector $u \in \mH_j$ we have $\left< T(u), u \right> = \lambda_j$, and so this eigenvalue can be computed with any such $u$. This fact will be applied in the proof of Theorem~\ref{thm:Toeplitz_radial}.
\end{remark}

\begin{remark}
\label{rmk:gamma_T_functions}
    With the above notation, every operator $T \in \End_H{\mH}$ is completely determined by either of the following
    \begin{itemize}
        \item the function $\gamma_T \in \ell^\infty(L)$, or
        \item the function $\widehat{\gamma}_T \in \ell^\infty(J)$ and the dimension function
            \begin{align*}
                d : J &\rightarrow \Z_+ \cup \{+\infty\} \\
                d(j) &= \dim \mH_j.
            \end{align*}
    \end{itemize}
    The latter provides the point spectrum and the multiplicity function for such spectrum. Clearly, the spectrum consists of eigenvalues only.
\end{remark}

\subsection{Bergman spaces and Toeplitz operators on the unit ball}
Following the conventions from \cite{Zhu}, for the unit ball $\B^n$ in $\C^n$ we let $\dif v$ denote the Lebesgue measure normalized so that $v(\B^n) = 1$. The usual Lebesgue measure will be denoted by $\dif z$. We will also denote by $\dif \sigma$ the volume element of $S^{2n-1}$ normalized so that $\sigma(S^{2n-1}) = 1$. In particular, we have (see \cite{Zhu})
\begin{align}
\label{eq:spherical}
    \dif v &= 2n r^{2n-1} \dif r \dif\sigma \\
    \dif z &= \frac{\pi^n}{n!} \dif v = \frac{2\pi^n}{(n-1)!} r^{2n-1} \dif r \dif\sigma. \notag
\end{align}
On the other hand, for every $\alpha > -1$ we consider the weighted measure
\[
        \dif v_\alpha(z) = c_\alpha (1 - |z|^2)^\alpha \dif v(z),
\]
where the constant
\[
    c_\alpha = \frac{\Gamma(n + \alpha + 1)}{n! \Gamma(\alpha + 1)} = \frac{1}{n B(n,\alpha+1)}
\]
is chosen so that $v_\alpha(\B^n) = 1$. The weighted Bergman space $\mH^2_\alpha(\B^n)$ is defined as the subspace of holomorphic functions that lie in $L^2(\B^n, v_\alpha)$. This is a closed subspace whose orthogonal projection $B_{\alpha}$ is given as follows
\begin{align*}
    B_{\alpha} : L^2(\B^n, v_\alpha) &\rightarrow \mH^2_\alpha(\B^n) \\
        (B_{\alpha}f)(z) &= \int_{\B^n} f(w)(1 - z \cdot \overline{w})^{-(n + \alpha + 1)} \dif v_\alpha(w),
\end{align*}
where the function
\begin{align*}
    \B^n \times \B^n &\rightarrow \C \\
        (z,w) &\mapsto (1 - z \cdot \overline{w})^{-(n + \alpha + 1)}
\end{align*}
is called the weighted Bergman kernel.

For every $a \in L^\infty(\B^n, \dif z)$ we define the Toeplitz operator $T_a$ on the weighted Bergman space $\mH^2_\alpha(\B^n)$ by
\begin{align*}
    T_a : \mH^2_\alpha(\B^n) &\rightarrow \mH^2_\alpha(\B^n) \\
        T_a f &= B_{\alpha}(af).
\end{align*}
In this case, $a$ is called the symbol of the Toeplitz operator $T_a$. It is easily seen that $T_a$ is a bounded operator with $\|T_a\| \leq \|a\|_\infty$.

On the other hand, the transformations that belong to the connected component of the identity of the biholomorphism group of $\B^n$ are given by the following action
\begin{align}
\label{eq:biholomorphisms}
    \SU(n,1) \times \B^n &\rightarrow \B^n \\
        \left(
        \begin{pmatrix}
            A & b \\
            c & d
        \end{pmatrix}, z
        \right)
        &\mapsto \frac{Az + b}{c\cdot z + d}, \notag
\end{align}
where $z \in \B^n$ is considered as a column, $A$ is an $n \times n$ matrix, $d$ is a complex number and the group $\SU(n,1)$ is defined by
\[
    \SU(n,1) = \{ M \in M_{n+1}(\C) \mid M I_{n,1} \overline{M}^t = I_{n,1} \},
\]
where
\[
    \begin{pmatrix}
        I_n & 0 \\
        0 & -1
    \end{pmatrix}.
\]
We note that the actual connected component of the biholomorphism group of $\B^n$ is the quotient of $\SU(n,1)$ by its center. However, it is easier to use the action of the group $\SU(n,1)$ for our purposes.

The isotropy subgroup of $0 \in \B^n$ for the action \eqref{eq:biholomorphisms} is the subgroup
\[
    \left\{
        \begin{pmatrix}
            A & 0 \\
            0 & b
        \end{pmatrix} \Big| A \in \U(n), b \in \T, \det(A) b = 1
    \right\}.
\]
Alternatively, the isotropy action at $0$ can be realized by the linear action of $\U(n)$ on $\B^n$ given by
\begin{align*}
    \U(n) \times \B^n &\rightarrow \B^n \\
        (A, z) &\mapsto Az,
\end{align*}
where $z$ is again considered as a column.

The subgroup of $\U(n)$ of diagonal matrices has elements of the form
\[
    \begin{pmatrix}
        t_1 & \cdots & 0 \\
        \vdots  & \ddots & \vdots   \\
        0 & \cdots & t_n
    \end{pmatrix},
\]
where $t_j \in \T$ for every $j = 1, \dots, n$. Hence, we will denote such subgroup by $\T^n$. The action of $\T^n$ on $\B^n$ is clearly given by
\begin{align*}
    \T^n \times \B^n &\rightarrow \B^n \\
        (t, z) &\mapsto (t_1 z_1, \dots, t_n z_n).
\end{align*}

The action of $\SU(n,1)$ on $\B^n$ yields actions on the weighted Bergman spaces. More precisely, for every $\alpha > -1$ we have a unitary representation
\begin{align*}
    \pi_{\alpha} : \widetilde{\SU}(n,1) &\rightarrow \U(\mH^2_\alpha(\B^n)) \\
        (\pi_{\alpha}(g)f)(z) &= j(g^{-1}, z)^{\frac{\alpha + n + 1}{n+1}} f(g^{-1} z),
\end{align*}
where $j(g, z)$ is the Jacobian at $z$ of the transformation of $\B^n$ induced by $g \in \widetilde{\SU}(n,1)$ and the lift of the action \eqref{eq:biholomorphisms} to $\widetilde{\SU}(n,1)$. In this construction it is essential to consider the universal covering group $\widetilde{\SU}(n,1)$ to ensure the existence of $j(g,z)^{\frac{\alpha + n + 1}{n + 1}}$ as a holomorphic function of both $g$ and $z$. These unitary representations define the holomorphic relative discrete series for the group $\widetilde{\SU}(n,1)$. It is worthwhile to understand the change of parameter for the discrete series that we made with respect to the representation theoretic notation as found, for example, in \cite{Wallach1}. For the latter, the parameter of the holomorphic discrete series is $\lambda > n$ and it is linearly related to our parameter $\alpha$ by the expression
\[
    \lambda = \alpha + n + 1.
\]

Note that the action of $\U(n)$ on $\B^n$ is volume preserving for the Lebesgue measure $\dif z$. It follows that the Jacobian $j(g,z) = 1$ at every $z \in \B^n$ when $g \in \widetilde{\SU}(n,1)$ projects to an element of $\U(n)$ (through the universal covering map of $\SU(n,1)$). Hence, for every $\alpha > -1$ the representation $\pi_\alpha$ restricts to a representation of $\U(n)$ (not just $\widetilde{\U}(n)$) and such restriction is given by
\[
    (\pi_\alpha(A)f)(z) = f(A^{-1} z),
\]
for every $A \in \U(n)$, $f \in \mH^2_\alpha(\B^n)$ and $z \in \B^n$. Clearly the same holds for the action of the subgroup $\T^n$. For simplicity, we will denote by the same symbol $\pi_\alpha$ these representations for both $\U(n)$ and $\T^n$.

The following result allows us to identify Toeplitz operators that intertwine the representation $\pi_\alpha$ restricted to a subgroup of $\U(n)$. We present the easy proof for the sake of completeness (see also \cite{DOQJFA}).

\begin{proposition}
    \label{prop:Toeplitz_intertwining}
    Let $H$ be closed subgroup of $\U(n)$. If $a \in L^\infty(\B^n,\dif z)$ is $H$-invariant, in other words, if it satisfies
    \[
        a \circ A = a
    \]
    for every $A \in H$, then, for every $\alpha > -1$, we have $T_a \in \End_H(\mH^2_\alpha(\B^n))$.
\end{proposition}
\begin{proof}
    First we note that, for every $\alpha > -1$, the measure $v_\alpha$ is $\U(n)$-invariant and so $H$-invariant as well. Hence, for every $f \in \mH^2_\alpha(\B^n)$ and $A \in H$ we have
    \begin{align*}
        T_a(A f)(z) &= \int_{\B^n} a(w) f(A^{-1} w) (1 - z \cdot \overline{w})^{-(n + \alpha + 1)} \dif v_\alpha(w) \\
            &= \int_{\B^n} a(A w) f(w) (1 - z \cdot \overline{A w})^{-(n + \alpha + 1)} \dif v_\alpha(w) \\
            &= \int_{\B^n} a(w) f(w) (1 - (A^{-1}z) \cdot \overline{w})^{-(n + \alpha + 1)} \dif v_\alpha(w) \\
            &= T_a(f)(A^{-1} z) = A \circ T_a(f)(z),
    \end{align*}
    thus implying that $T_a \circ A = A \circ T_a$.
\end{proof}

Let us denote by $\mP(\C^n)$ the algebra of polynomial functions on $\C^n$. Since $v_\alpha$ is a probability measure, it follows that $\mP(\C^n) \subset \mH^2_\alpha(\B^n)$ for every $\alpha > -1$. Furthermore, the following result is well known (see for example \cite{Wallach1}) and it will be essential for our constructions.

\begin{proposition}
    \label{prop:Poly_density_invariance}
    The space $\mP(\C^n)$ is dense and $\U(n)$-invariant in $\mH^2_\alpha(\B^n)$ for every $\alpha > -1$.
\end{proposition}

\section{$\T^n$-intertwining operators}\label{sec:Tn_intertwining}
Let us consider the algebra $\End_{\T^n}(\mH^2_\alpha(\B^n))$ of bounded operators on $\mH^2_\alpha(\B^n)$ that intertwine the representation of $\T^n$. Our goal is to establish the commutativity of such algebra by realizing it as an algebra of multiplication operators. Our main tool is the following well known result, whose proof we include for completeness. In what follows we will use without further mention the multi-index notation for polynomials (see \cite{Zhu}).

\begin{proposition}
    \label{prop:Poly_Tn}
    The decomposition of $\mP(\C^n)$ into irreducible $\T^n$-modules is given by
    \[
        \mP(\C^n) = \sum_{m \in \N^n} \C z^m.
    \]
    More precisely, for every $m \in \N^n$, the space $\C z^m$ is an irreducible $\T^n$-submodule, and we also have $\C z^m \not\cong \C z^{m'}$ as $\T^n$-modules and $\C z^m \perp \C z^{m'}$ whenever $m \not= m'$. In particular, for every $\alpha > -1$ we have
    \[
        \mH^2_\alpha(\B^n) = \bigoplus_{m \in \N^n} \C z^m
    \]
    as an orthogonal direct sum of Hilbert spaces that yields the decomposition of $\mH^2_\alpha(\B^n)$ into irreducible $\T^n$-modules.
\end{proposition}
\begin{proof}
    The first sum in the statement holds trivially and since every space $\C z^m$, for $m \in \N^n$, is $1$-dimensional and $\T^n$-invariant such sum yields a decomposition into irreducible $\T^n$-modules. The orthogonality of these $1$-dimensional subspaces is well known (see \cite{Zhu}). It remains to prove that $\C z^m \not\cong \C z^{m'}$ as $\T^n$-modules when $m \not= m'$.

    For every $t \in \T^n$ and $m \in \N^n$ we have
    \[
        t\cdot z^m = (t^{-1} z)^m = t_1^{-m_1}\cdot \ldots \cdot t_n^{-m_n} z^m,
    \]
    and so $\C z^m$ is an irreducible $\T^n$-module with character given by
    \[
        \chi_m(t) = t_1^{-m_1}\cdot \ldots \cdot t_n^{-m_n}.
    \]
    Since the isomorphism class of an irreducible $\T^n$-module is determined by its character, the claim that $\C z^m \not\cong \C z^{m'}$ as $\T^n$-modules when $m \not= m'$ is now clear.
\end{proof}

Proposition~\ref{prop:Poly_Tn} allows us to apply the results from Section~\ref{sec:Preliminaries}. We do so by choosing, for every $m \in \N^n$, the monomial
\[
    f_m(z) = z^m.
\]
We recall (see \cite{Zhu}) that for every $\alpha > -1$ we have
\[
    \|f_m\|^2_\alpha = \frac{m! \Gamma(n + \alpha + 1)}{\Gamma(n + |m| + \alpha + 1)},
\]
which yields the following well known orthonormal basis for $\mH^2_\alpha(\B^n)$
\[
    \left(
        e_m(z) =  \sqrt{\frac{\Gamma(n + |m| + \alpha + 1)}{m! \Gamma(n + \alpha + 1)}} z^m
    \right)_{m \in \N^n}.
\]
With this choice of polynomials, Propositions~\ref{prop:End_H(mH)-commutative}, \ref{prop:End_H(mH)-realization} and \ref{prop:Poly_Tn} yield the following result.

\begin{theorem}
    \label{thm:End_Tn}
    For every $\alpha > -1$, the algebra $\End_{\T^n}(\mH^2_\alpha(\B^n))$ is commutative. More precisely, with the above notation and for the unitary map
    \begin{align*}
        R : \mH^2_\alpha(\B^n) &\rightarrow \ell^2(\N^n) \\
            R(f) &= \left( \left< f, e_m\right>_\alpha\right)_{m \in \N^n},
    \end{align*}
    every operator $T \in \End_{\T^n}(\mH^2_\alpha(\B^n))$ is unitarily equivalent to $R T R^*$ which is the multiplication operator on $\ell^2(\N^n)$ by the function
    \begin{align*}
        \gamma_T : \N^n &\rightarrow \C \\
            \gamma_T(m) &= \left< T(e_m), e_m \right>_\alpha.
    \end{align*}
\end{theorem}

As a consequence of Proposition~\ref{prop:End_H(mH)-realization} and Theorem~\ref{thm:End_Tn} we obtain the following result.

\begin{corollary}
    \label{cor:End_Tn}
    With the above notation, the assignment
    \[
        T \mapsto \gamma_T
    \]
    defines an isomorphism of algebras $\End_{\T^n}(\mH^2_\alpha(\B^n)) \rightarrow \ell^\infty(\N^n)$.
\end{corollary}

\section{$\U(n)$-intertwining operators}\label{sec:U(n)_intertwining}
Let us now consider the algebra $\End_{\U(n)}(\mH^2_\alpha(\B^n))$ of bounded operators on $\mH^2_\alpha(\B^n)$ that intertwine the representation of $\U(n)$. This time, the section's goal is to prove the commutativity of this algebra by realizing it as an algebra of multiplication operators. The main ingredient to achieve this is the following well known result. Again, we present the proof for the sake of completeness.

\begin{proposition}
    \label{prop:Poly_U(n)}
    Let us denote by $\mP^k(\C^n)$ the space of homogeneous polynomials on $\C^n$ of degree $k$. Then, the decomposition of $\mP(\C^n)$ into irreducible $\U(n)$-modules is given by
    \[
        \mP(\C^n) = \sum_{k \in \N} \mP^k(\C^n).
    \]
    More precisely, for every $k \in \N$, the space $\mP^k(\C^n)$ is an irreducible $\U(n)$-submodule, and we also have $\mP^k(\C^n) \not\cong \mP^l(\C^n)$ as $\U(n)$-modules and $\mP^k(\C^n) \perp \mP^l(\C^n)$ whenever $k \not= l$. In particular, for every $\alpha > -1$ we have
    \[
        \mH^2_\alpha(\B^n) = \bigoplus_{k \in \N} \mP^k(\C^n)
    \]
    as an orthogonal direct sum of Hilbert spaces that yields the decomposition of $\mH^2_\alpha(\B^n)$ into irreducible $\U(n)$-modules.
\end{proposition}
\begin{proof}
    Consider the general linear group on $\C^n$ which is denoted by $\GL(n,\C)$. We recall that $\U(n)$ is the set of real points for a suitable algebraic structure of $\GL(n, \C)$ over $\R$. In particular, $\U(n)$ is Zariski dense in $\GL(n,\C)$. The group $\GL(n,\C)$ acts by the following expression on the space of complex polynomials
    \begin{align*}
        \GL(n,\C) \times \mP(\C^n) &\rightarrow \mP(\C^n) \\
            (Af)(z) &= f(A^{-1}z),
    \end{align*}
    whose restriction to $\U(n)$ is precisely the representation $\pi_\alpha$ of $\U(n)$ on polynomial functions. Clearly, this $\GL(n,\C)$-action leaves invariant every subspace $\mP^k(\C^n)$. Furthermore, the representation of $\GL(n,\C)$ on the (finite dimensional) space $\mP^k(\C^n)$ is rational in the sense of algebraic groups (see \cite{GoodmanWallach}). Thus, the Zariski density of $\U(n)$ in $\GL(n,\C)$ implies that a subspace of $\mP^k(\C^n)$ is $\GL(n,\C)$-invariant if and only if it is $\U(n)$-invariant. In particular, the decompositions into irreducible submodules of each $\mP^k(\C^n)$, and so of $\mP(\C^n)$, with respect to either $\GL(n,\C)$ or $\U(n)$ are the same.

    On the other hand, one can use the Schur-Weyl duality of $\GL(n,\C)\times \C^*$ to determine the decomposition of each $\mP^k(\C^n)$ into irreducible submodules. Applying the results in Section~5.6.2 from \cite{GoodmanWallach} (see for example Theorem~5.6.7 of this reference) it follows that $\mP^k(\C^n)$ is an irreducible $\GL(n,\C)$-module. That these irreducible modules are mutually non-isomorphic follows from the fact that they all have different dimension.

    Next, we note that the orthogonality of the subspaces $\mP^k(\C^n)$ follows since these are irreducible mutually non-isomorphic and the $\U(n)$-representation is unitary.
\end{proof}

Proposition~\ref{prop:Poly_U(n)} allows us to apply the results from Section~\ref{sec:Preliminaries}. We will consider, for every $\alpha > -1$, the same Hilbert base $(e_m)_{m \in \N^n}$ defined in Section~\ref{sec:Tn_intertwining}. Hence, Propositions~\ref{prop:End_H(mH)-commutative}, \ref{prop:End_H(mH)-realization} and \ref{prop:Poly_U(n)} yield the following result.

\begin{theorem}
    \label{thm:End_U(n)}
    For every $\alpha > -1$, the algebra $\End_{\U(n)}(\mH^2_\alpha(\B^n))$ is commutative. More precisely, for the above notation and for the unitary map
    \begin{align*}
        R : \mH^2_\alpha(\B^n) &\rightarrow \ell^2(\N^n) \\
            R(f) &= \left( \left< f, e_m \right>_\alpha \right)_{m \in \N^n},
    \end{align*}
    every operator $T \in \End_{\U(n)}(\mH^2_\alpha(\B^n))$ is unitarily equivalent to $R T R^*$ which is the multiplication operator on $\ell^2(\N^n)$ by the function
    \begin{align*}
        \gamma_T : \N^n &\rightarrow \C \\
            \gamma_T(m) &= \left< T(e_m), e_m \right>_\alpha.
    \end{align*}
    Furthermore, let us choose for every $k \in \N$ a unitary vector $u_k \in \mP^k(\C^n)$ and consider the function
    \begin{align*}
        \widehat{\gamma}_T : \N &\rightarrow \C \\
            \widehat{\gamma}_T(k) &= \left< T(u_k), u_k \right>_\alpha.
    \end{align*}
    Then, we have
    \[
        \widehat{\gamma}_T(|m|) = \gamma_T(m),
    \]
    for every $m \in \N^n$. In particular, $\gamma_T(m) = \gamma_T(m')$ whenever $m, m' \in \N^n$ and $|m| = |m'|$.
\end{theorem}
\begin{proof}
    The claims involving $R$ follow directly from Proposition~\ref{prop:End_H(mH)-realization}. By the last claim of such proposition, it also follows that for every $k \in \N$, the function $\gamma_T$ is constant on the set of values $m \in \N^n$ for which $e_m \in \mP^k(\B^n)$, in other words, when $|m| = k$.

    On the other hand, as in the proof of Proposition~\ref{prop:End_H(mH)-commutative} and by Proposition~\ref{prop:Poly_U(n)}, Schur's Lemma implies that $T_a$ acts by a scalar multiple on $\mP^k(\C^n)$ for every $k \in \N$. In particular, we have
    \[
        \left< T_a u_k, u_k \right>_\alpha = \left< T_a e_m, e_m \right>_\alpha,
    \]
    whenever $k = |m|$.
\end{proof}

As a consequence of Proposition~\ref{prop:End_H(mH)-realization} and Theorem~\ref{thm:End_U(n)} we obtain the following result.

\begin{corollary}
    \label{cor:End_U(n)}
    With the above notation, the assignment
    \[
        T \mapsto \widehat{\gamma}_T
    \]
    defines an isomorphism of algebras $\End_{\U(n)}(\mH^2_\alpha(\B^n)) \rightarrow \ell^\infty(\N)$.
\end{corollary}

\section{Toeplitz operators with separately radial symbols}\label{sec:Toeplitz_separately_radial}
Following \cite{QVBall1}, we say that a function $a \in L^\infty(\B^n, \dif z)$ is separately radial if it satisfies
\[
    a(z) = a(z_1, \dots, z_n) = a(|z_1|, \dots, |z_n|)
\]
for almost every $z \in \B^n$. In other words, the function $a$ is separately radial if and only if it is $\T^n$-invariant. The corresponding Toeplitz operator is thus called a separately radial Toeplitz operator. By Proposition~\ref{prop:Toeplitz_intertwining} it follows that a separately radial Toeplitz operator is $\T^n$-invariant in every Bergman space $\mH^2_\alpha(\B^n)$. These remarks and Theorem~\ref{thm:End_Tn} allow us to obtain the following result.

\begin{theorem}
    \label{thm:Toeplitz_separately_radial}
    The $C^*$-algebra generated by Toeplitz operators with separately radial symbols is commutative.

    More precisely, for every $\alpha > -1$, let us consider the orthonormal base $(e_m)_{m \in \N^n}$ of $\mH^2_\alpha(\B^n)$ where
    \[
        e_m(z) = \sqrt{\frac{\Gamma(n + |m| + \alpha + 1)}{m! \Gamma(n + \alpha + 1)}} z^m
    \]
    for every $z \in \B^n$ and $m \in \N^n$, and let us define the unitary map
    \begin{align*}
        R : \mH^2_\alpha(\B^n) &\rightarrow \ell^2(\N^n) \\
            R(f) &= \left(\left< f, e_m \right>_\alpha \right)_{m \in \N^n}.
    \end{align*}
    Then, for every separately radial symbol $a \in L^\infty(\B^n)$, the Toeplitz operator $T_a$ is unitarily equivalent to the multiplication operator $R T_a R^* = \gamma_{a,\alpha} I$ where the function $\gamma_{a,\alpha} \in \ell^\infty(\N^n)$ is given by
    \begin{align*}
        \gamma_{a,\alpha}(m) &= \left< T_a e_m, e_m \right>_\alpha
                = \left<a e_m, e_m \right>_\alpha \\
                &= \frac{2^n \Gamma(n + |m| + \alpha + 1)}{m!\Gamma(\alpha + 1)} \int_{\tau(\B^n)} a(r) r^{2m} (1 - r^2)^\alpha \prod_{j=1}^n r_j \dif r_j \\
                &= \frac{\Gamma(n + |m| + \alpha + 1)}{m!\Gamma(\alpha + 1)} \int_{\Delta(\B^n)} a(\sqrt{r}) r^m (1 - (r_1 + \dots + r_n))^\alpha \dif r,
    \end{align*}
    for every $m \in \N^n$, where $\Delta(\B^n)$ is the set of point $r \in \R^n$ such that $r_1 + \dots + r_n < 1$ and $r_j \geq 0$, for every $j = 1, \dots, n$, and $\sqrt{r} = (\sqrt{r}_1, \dots, \sqrt{r}_n)$.
\end{theorem}
\begin{proof}
    By Proposition~\ref{prop:Toeplitz_intertwining} and Theorem~\ref{thm:End_Tn} it is enough to compute
    \begin{align*}
        \left< a e_m, e_m \right>_\alpha
            &= c_\alpha \int_{\B^n} a(z) |e_m(z)|^2 (1 - |z|^2)^\alpha \dif v(z) \\
            &= \frac{c_\alpha n!}{\|z^m\|^2_\alpha \pi^n}
                \int_{\B^n} a(|z_1|, \dots,|z_n|) |z^m|^2 (1 - |z|^2)^\alpha \dif z \\
        \intertext{and using polar coordinates in each axis of $\C^n$ we obtain}
            &= \frac{2^n \Gamma(n + |m| + \alpha + 1)}{m!\Gamma(\alpha + 1)}
                \int_{\tau(\B^n)} a(r) r^{2m} (1 - r^2)^\alpha \prod_{j=1}^n r_j \dif r_j.
    \end{align*}
    The last identity is obtained by applying the change of coordinates $r \mapsto r^2$.
\end{proof}

\begin{remark}
    We note that the formulas of the previous result are exactly the same as those found in Theorem~10.1 from \cite{QVBall1} and computed in Theorem~3.1 from \cite{QVReinhardt}.
\end{remark}

As a consequence of Theorem~\ref{thm:Toeplitz_separately_radial} we obtain the following orthogonality relations.

\begin{corollary}
    \label{cor:orthogonality_relations_Tn}
    If $a \in L^\infty(\B^n, \dif z)$ is a separately radial symbol, then for every $\alpha > -1$ we have
    \begin{align*}
        &\left< T_a z^m, z^{m'} \right>_\alpha  \\
        &= \left< a z^m, z^{m'} \right>_\alpha =
                \begin{cases}
            0 &\text{if } m \not= m' \\
            \displaystyle\int_{\Delta(\B^n)} a(\sqrt{r}) r^m (1 - (r_1 + \dots + r_n))^\alpha \dif r
            &\text{if } m = m'
        \end{cases},
    \end{align*}
    for every $m, m' \in \N^n$.
\end{corollary}

\begin{remark}
    Note that in the previous corollary another formula can be obtained from the expression of $\gamma_{a,\alpha}$ in Theorem~\ref{thm:Toeplitz_separately_radial} that involves integration over $\tau(\B^n)$.
\end{remark}

\section{Toeplitz operators with radial symbols}\label{sec:Toeplitz_radial}
Following \cite{GKV}, we say that a function $a \in L^\infty(\B^n, \dif z)$ is radial if it satisfies
\[
    a(z) = a(|z|)
\]
for almost every $z \in \B^n$. Hence, the function $a$ is radial if and only if it is $\U(n)$-invariant. A Toeplitz operator is called radial if its symbol is radial. By Proposition~\ref{prop:Toeplitz_intertwining} it follows that a radial Toeplitz operator is $\U(n)$-invariant in every Bergman space $\mH^2_\alpha(\B^n)$. We will now apply Theorem~\ref{thm:End_U(n)} to obtain the structure of radial Toeplitz operators. This requires the choice of unitary vectors on each term in the decomposition of $\mH^2_\alpha(\B^n)$ into irreducible $\U(n)$-modules, in other words, we need to choose and normalize a non-zero element in each subspace $\mP^k(\C^n)$.

For each $k \in \N$, let us choose the element of $\mP^k(\C)$ given by
\[
    f_k(z) =
        \sum_{\substack{m \in \N^n \\ |m| = k}} \sqrt{\binom{k}{m}} z^m,
\]
where the multinomial coefficient is defined by
\[
    \binom{k}{m} = \frac{k!}{m_1! \cdot \ldots \cdot m_n!},
\]
for every $m \in \N^n$ such that $|m| = k$.

\begin{lemma}
    \label{lem:f_k_norm}
    For every $\alpha > -1$ and for the above notation we have
    \[
        \|f_k\|^2_\alpha = nc_\alpha B(n+k,\alpha+1) = \frac{B(n+k,\alpha+1)}{B(n,\alpha+1)}.
    \]
\end{lemma}
\begin{proof}
    We compute
    \begin{align*}
        \|f_k\|^2_\alpha &= c_\alpha \int_{\B^n} |f_k(z)|^2 (1 - |z|^2)^\alpha \dif v(z) \\
                &= c_\alpha \sum_{\substack{m, m' \in \N^n \\ |m| = |m'| = k}}
                    \sqrt{\binom{k}{m}\binom{k}{m'}}
                    \int_{\B^n}z^m \overline{z}^{m'} (1 - |z|^2)^\alpha \dif v(z), \\
        \intertext{by the orthogonality of the monomials in the Bergman spaces (see \cite{Zhu}) we have}
                &= c_\alpha \sum_{\substack{m \in \N^n \\ |m| = k}}
                    \binom{k}{m}
                    \int_{\B^n}z^m \overline{z}^{m} (1 - |z|^2)^\alpha \dif v(z) \\
                &= c_\alpha \int_{\B^n}
                    \sum_{\substack{m \in \N^n \\ |m| = k}}
                        \binom{k}{m} |z_1|^{2m_1} \cdot \ldots \cdot |z_n|^{2m_n}
                            (1 - |z|^2)^\alpha \dif v(z) \\
                &= c_\alpha \int_{\B^n} (|z_1|^2 + \dots + |z_n|^2)^k (1 - |z|^2)^\alpha \dif v(z), \\
        \intertext{and with respect to spherical coordinates we obtain from \eqref{eq:spherical}}
                &= 2nc_\alpha \int_0^1 \int_{S^{2n-1}} r^{2k} (1 - r^2)^\alpha r^{2n-1}\dif r d\sigma \\
                &= 2nc_\alpha \int_0^1 r^{2n + 2k - 1} (1 - r^2)^\alpha \dif r \\
                &= nc_\alpha B(n+k,\alpha+1) \\
                &= \frac{B(n+k,\alpha+1)}{B(n,\alpha+1)}.
    \end{align*}
\end{proof}

For every $\alpha > -1$, we consider the set of polynomials given by
\[
    \left(
        u_k = \frac{f_k}{\|f_k\|_\alpha}
    \right)_{k \in \N}
\]
which is orthonormal in $\mH^2_\alpha(\B^n)$, for every $\alpha > -1$, and that has exactly one element in each space $\mP^k(\C^n)$.

\begin{theorem}
    \label{thm:Toeplitz_radial}
    The $C^*$-algebra generated by Toeplitz operators with radial symbols is commutative.

    More precisely, for every $\alpha > -1$, let us consider the orthonormal base $(e_m)_{m \in \N^n}$ of $\mH^2_\alpha(\B^n)$ defined in Theorem~\ref{thm:Toeplitz_separately_radial}, and the unitary operator
    \begin{align*}
        R : \mH^2_\alpha(\B^n) &\rightarrow \ell^2(\N^n) \\
            R(f) &= \left(\left< f, e_m \right>_\alpha \right)_{m \in \N^n}.
    \end{align*}
    Then, for every radial symbol $a \in L^\infty(\B^n, \dif z)$, the Toeplitz operator $T_a$ is unitarily equivalent to the multiplication operator $R T_a R^* = \gamma_{a,\alpha} I$ where $\gamma_{a,\alpha} \in \ell^\infty(\N^n)$ is given by
    \[
        \gamma_{a,\alpha}(m) = \left< T_a e_m, e_m \right>_\alpha
                = \left<a e_m, e_m \right>_\alpha
    \]
    for every $m \in \N^n$. This function satisfies $\gamma_{a,\alpha}(m) = \gamma_{a,\alpha}(m')$, for $m, m' \in \N^n$ such that $|m| = |m'|$, and so it determines a function
    \begin{align*}
        \widehat{\gamma}_{a,\alpha} : \N &\rightarrow \C \\
            \widehat{\gamma}_{a,\alpha}(|m|) &= \gamma_{a,\alpha}(m)
    \end{align*}
    where $m \in \N^n$. Furthermore, we have
    \begin{align*}
        \widehat{\gamma}_{a,\alpha}(k)
            &= \frac{2\displaystyle\int_0^1 a(r) r^{2n + 2k - 1}(1 - r^2)^\alpha \dif r}{B(n+k,\alpha+1)} \\
            &= \frac{\displaystyle\int_0^1 a(\sqrt{r}) r^{n + k - 1}(1 - r)^\alpha \dif r}{B(n+k,\alpha+1)},
    \end{align*}
    for every $k \in \N$.
\end{theorem}
\begin{proof}
    By the previous remarks, Theorem~\ref{thm:End_U(n)} and our choice of unitary vectors $(u_k)_{k \in \N}$ it is enough to compute the following
    \begin{align*}
        \left< a u_k, u_k \right>_\alpha
            &= c_\alpha \int_{\B^n} a(z) |u_k(z)|^2 (1 - |z|^2)^\alpha \dif v(z) \\
            &= \frac{c_\alpha}{\|f_k\|^2_\alpha}
                    \int_{\B^n} a(z)
                    \sum_{\substack{m \in \N^n \\ |m| = k}} \sqrt{\binom{k}{m}} z^m
                    \sum_{\substack{m' \in \N^n \\ |m'| = k}} \sqrt{\binom{k}{m'}} \overline{z}^{m'}
                        (1 - |z|^2)^\alpha \dif v(z) \\
            &= \frac{c_\alpha}{\|f_k\|^2_\alpha}
                    \sum_{\substack{m, m' \in \N^n \\ |m| = |m'| = k}} \sqrt{\binom{k}{m}\binom{k}{m'}}
                        \int_{\B^n} a(z) z^m \overline{z}^{m'} (1 - |z|^2)^\alpha \dif v(z), \\
        \intertext{we now recall (see \cite{Zhu}) that the orthogonality of the monomials $z^m$ depends on the invariance of the measure $v_\alpha$ with respect to $\U(n)$, and since $a(z)(1 - |z|^2)^\alpha \dif v(z)$ satisfies this same invariance we have}
            &= \frac{c_\alpha}{\|f_k\|^2_\alpha}
                    \sum_{\substack{m \in \N^n \\ |m| = k}} \binom{k}{m}
                        \int_{\B^n} a(z) z^m \overline{z}^m (1 - |z|^2)^\alpha \dif v(z) \\
            &= \frac{c_\alpha}{\|f_k\|^2_\alpha}
                    \int_{\B^n} a(z) \sum_{\substack{m \in \N^n \\ |m| = k}} \binom{k}{m}
                        |z_1|^{2m_1} \dots |z_n|^{2m_n}(1 - |z|^2)^\alpha \dif v(z) \\
            &= \frac{c_\alpha}{\|f_k\|^2_\alpha}
                    \int_{\B^n} a(z) |z|^{2k} (1 - |z|^2)^\alpha \dif v(z), \\
        \intertext{introducing spherical coordinates we obtain}
            &= \frac{2nc_\alpha}{\|f_k\|^2_\alpha}
                    \int_0^1 a(r) r^{2n + 2k - 1} (1 - r^2)^\alpha \dif r, \\
        \intertext{and applying Lemma~\ref{lem:f_k_norm} we obtain}
            &= \frac{2\displaystyle\int_0^1 a(r) r^{2n + 2k - 1} (1 - r^2)^\alpha \dif r}{B(n+k,\alpha+1)}.
     \end{align*}
     This provides the value of $\widehat{\gamma}_{a,\alpha}(k)$ for $k \in \N$. The last identity in the statement is obtained by applying the change of coordinates $r \mapsto r^2$.
\end{proof}

\begin{remark}
    Our formulas are similar to formula (3.1) found in Theorem~3.1 from \cite{GKV}. In fact, from following the definition of the coefficients considered in \cite{GKV} it is possible to prove that they are exactly the same. However, our approach has provided an interpretation of such coefficients and the  expression for $\widehat{\gamma}_{a,\alpha}$: the function $\widehat{\gamma}_{a,\alpha}$ of the multiplication operator equivalent to $T_a$ is computed as the orthogonal projections of the values of $T_a$ onto the irreducible components of $\mH^2_\alpha(\B^n)$ with respect to the representation of the group $\U(n)$. Moreover, this representation theoretic approach has in fact allowed us to obtain our more explicitly presented formulas.
\end{remark}

From Theorem~\ref{thm:Toeplitz_radial} we obtain the following orthogonality relations.

\begin{corollary}
    \label{cor:orthogonality_relations_U(n)}
    If $a \in L^\infty(\B^n,\dif z)$ is a radial symbol, then for every $\alpha > -1$ we have
    \begin{align*}
        &\left< T_a z^m, z^{m'} \right>_\alpha  \\
        &= \left< a z^m, z^{m'} \right>_\alpha  \\
        &=
            \begin{cases}
                0 &\text{if } |m| \not= |m'| \\
                \displaystyle\frac{\Gamma(n + |m| + \alpha + 1)}{m! \Gamma(n + \alpha + 1)B(n+k,\alpha+1)}
                \displaystyle\int_0^1 a(\sqrt{r}) r^{n + k - 1}(1 - r)^\alpha \dif r &\text{if } |m| = |m'|
            \end{cases},
    \end{align*}
    for every $m, m' \in \N^n$.
\end{corollary}
\begin{proof}
    By the proof of Proposition~\ref{prop:End_H(mH)-commutative}, Schur's Lemma implies that $T_a$ preserves the spaces $\mP^k(\C^n)$ for every $k \in \N$. Moreover, by Proposition~\ref{prop:Poly_U(n)} these spaces are mutually orthogonal. This implies the first case.

    For the second case, we use again the application of Schur's Lemma in Proposition~\ref{prop:Poly_U(n)}. This shows that the value $\left< T_a u, u \right>_\alpha$ is the same for every unitary vector $u \in \mP^k(\C^n)$. In particular, we have
    \[
        \left< T_a e_m, e_m \right>_\alpha = \left< T_a u_{|m|}, u_{|m|} \right>_\alpha
    \]
    for every $m \in \N^n$. And the result follows by writing down $e_m$ in terms of $z^m$ and the formulas from Theorem~\ref{thm:Toeplitz_radial}.
\end{proof}

\begin{remark}
    Note that in the previous corollary another formula can be obtained from the expression of $\widehat{\gamma}_{a,\alpha}$ in Theorem~\ref{thm:Toeplitz_radial} that uses $a(r)$ instead of $a(\sqrt{r})$.
\end{remark}

\end{document}